\documentclass[amscd,amssymb,verbatim,12pt]{amsart}
\usepackage{epsfig}
\usepackage{graphicx}

\newif\ifAMS
\IfFileExists{amssymb.sty}
  {\AMStrue\usepackage{amssymb}}
  {\usepackage{latexsym}}
  \usepackage{enumerate}
\theoremstyle{plain}

\newtheorem*{Menger}{Menger's Theorem}
\newtheorem*{Cactus}{Theorem \ref{cactus}}
\newtheorem{Thm}{Theorem}[section]
\newtheorem{Cor}[Thm]{Corollary}
\newtheorem{Lem}[Thm]{Lemma}

\theoremstyle{definition}
\newtheorem{Def}{Definition}[section]

\theoremstyle{remark}
\newtheorem{Rem}{Remark}

\newtheorem{Ex}[Thm]{Example}

\newcommand{\interior}{^{ \kern-5pt ^\circ}}
\newcommand {\bd}{\partial}

\newcommand {\cP}{{\mathcal  P}}

\begin{document}
\title
{A cactus theorem for end cuts}

\author
{Anastasia Evangelidou and Panos Papasoglu }


\email [Anastasia Evangelidou]{aneva@spidernet.com.cy  } \email [Panos
Papasoglu]{papazoglou@maths.ox.ac.uk}

\address
[Anastasia Evangelidou] { Mathematics Department, University of
Athens, Athens 157 84, Greece}
\address
[Panos Papasoglu] 
{Mathematical Institute, University of Oxford, 24-29 St Giles', Oxford, OX1 3LB, U.K. }

\begin{abstract} Dinits-Karzanov-Lomonosov showed that it is possible to encode all
minimal edge cuts of a graph by a tree-like structure called a
cactus. We show here that minimal edge cuts separating ends of the graph
rather than vertices can be `encoded' also by a cactus.
We apply our methods to finite graphs as well and we show that several types of cuts 
can be encoded by cacti.

\end{abstract}
\maketitle
\section{Introduction and Preliminaries}
Vertex and edge cuts of graphs have been studied extensively in several
different contexts: graph theory, geometric group theory, topology
and networks. They have played an important role in applications,
notably in clustering algorithms, combinatorial optimization and
network design.

E.A. Dinits,  A.V. Karzanov, M.V. Lomonosov \cite{DKL} (see also
\cite{KT}, sections 7.4,7.5) gave an elegant way to encode all
minimal edge cuts of a graph by a cactus, a tree-like structure.
For a recent short proof of their theorem see \cite{FF}. This
structure theorem has found many important applications (\cite
{NV}, \cite {NGM}). The crucial observation in \cite{DKL} is that minimal edge cuts
which `cross' have a circular structure.

Tutte has studied vertex cuts and has shown that minimal vertex cuts of cardinality 2
can be encoded by a tree like structure (\cite{T1} Ch. IV, \cite{T2} ch. 11, \cite {T3}). In fact one can see
Tutte's theorem as a cactus theorem for vertex cuts, but his theorem applies only to cuts of cardinalitty 2.
In \cite{DSS} Tutte's theorem was extended to infinite, locally finite graphs.

There is a similar theory of cuts of connected metric spaces (\cite{W}, \cite{B})
dealing with cut points and cut pairs. In particular in the case of cut pairs Bowditch shows
that crossing cut pairs have a circular structure.

Stallings \cite {St} (in the locally finite case) and Dunwoody
\cite{D} (in general) have shown that if $\Gamma $ is a graph with
more than one end then there is a set of minimal end cuts of
$\Gamma $ which is invariant under $Aut(\Gamma)$ and which can be
encoded by a tree. The main motivation of Stallings and Dunwoody was to
classify groups with many ends.

In this paper we show that one can encode the set of \emph{all} minimal end cuts of
a graph by a cactus. We note that Stallings and Dunwoody proceeded by finding a subset of minimal
end cuts which can be encoded by a tree (and is invariant under the automorphism group) while our work reveals
that the set of all minimal end cuts has the finer structure of a cactus.
In particular we show that crossing end cuts have a circular structure.
It follows from our result too that groups with many ends split over finite groups.

Let $\Gamma=(V,E)$ be a connected graph. 
A subset $K\subset E$ is
called an {\it edge cut} if $\Gamma -K=(V,E-K)$ has at least two
connected components. A subset $K\subset V$ is a \textit{vertex cut} of
$\Gamma $ if $\Gamma -K$ is not connected. If $A,B\subset \Gamma $ we say that
$K$ \textit{separates} $A$ from $B$ if any path joining a vertex of $A$ to a vertex of $B$ intersects $K$.

A {\it ray} of $\Gamma $ is an infinite sequence of distinct
consecutive vertices $v_0,v_1,v_2,...$ of $\Gamma $. We say that
two rays $r_1,r_2$ are equivalent if for any finite edge cut $K$
all vertices of $r_1\cup r_2$ except finitely many are contained
in the same connected component of $\Gamma -K$. The \textit{ends} of $\Gamma
$ are the equivalence classes of rays. If $A\subset \Gamma $ and $e$ is an end we say that
$A$ \textit{contains} $e$ is almost all vertices of some (all) ray $r$ representing $e$ are contained in $A$.
Let $K$ be a finite edge
cut of $\Gamma $. We say that $K$ is an \textit{end cut} of
$\Gamma $ if there are at least two connected components of
$\Gamma -K$ which contain rays. We say that an end cut is a
\textit{minimal cut} of $\Gamma $ if its cardinality is minimal
among all end cuts of $\Gamma $. We remark that if $K $ is a
minimal cut then $\Gamma-K$ has exactly two components. We say
that a minimal cut $K$ \textit{separates} two ends $e_1,e_2$ if
there are two rays $r_1,r_2$ representing respectively $e_1,e_2$
such that $r_1,r_2$ are contained in distinct connected components
of $\Gamma -K$. We say that two minimal cuts $K,L$ are \textit{equivalent}
if any two ends $e_1,e_2$ are separated by $K$ if and only if they are separated
by $L$. We denote the equivalence class of $K$ by $[K]$.
If $M$ is a subset of a graph we denote by $V(M)$
the set of vertices of $M$. We say that a set of consecutive edges
of a graph $e_1=[a_1,a_2],...,e_k=[a_k,a_{k+1}]$ is a
\textit{cycle} if $a_1=a_{k+1}$ and for $i\neq j$, $a_i\neq a_j$
if $\{i,j\}\neq \{1,k+1\}$.
A graph $\mathcal{C}=(V,E)$ is called a \textit{cactus} if any two cycles in $\mathcal{C}$
have at most one vertex in common. We say that a vertex $v$ of $\mathcal{C}$ is an \textit{end-vertex}
if $\mathcal{C}-v$ is connected. We remark that the degree of an end vertex is 1 or 2. We state now the main result of this paper:

\begin{Cactus} Let $\Gamma $ be a graph. Then there is a cactus $\mathcal{C}$, an onto map $f$ from the ends
of $\Gamma $ to the union of the ends of $\mathcal{C}$ with the end vertices of $\mathcal{C}$ and a 1-1 and onto map $g$ from equivalence classes
of minimal cuts of $\Gamma $ to the minimal edge cuts of $\mathcal{C}$ so that two ends $e_1,e_2$ of $\Gamma $ are separated by a minimal
cut $K$ if and only if $f(e_1),f(e_2)$ are separated by $g([K])$. Moreover any automorphism of $\Gamma $ induces an automorphism of $\mathcal{C}$.
\end{Cactus}

It turns out that one can show similar results for the structure of `small' edge cuts of a finite graph if the finite
graph contains `big' subgraphs that can not be cut by `small' cuts. We discuss this briefly in section \ref{gen}.

Martin Dunwoody brought to our attention his work with Kr\"{o}n \cite{DK} which contains some arguments similar to the ones used here.
He also told us about Tutte's work and this led us to consider edge cuts of finite graphs as well.
We would like to thank Aggelos Georgakopoulos
for pointing out a mistake in an earlier version of this paper.

\section{Pretrees}

We will use the notion of pretrees (\cite {B},\cite{B2}) to show that the
set of minimal cuts can be represented by a cactus.

Informally a pretree can be thought of as a subset of a tree.
Given a pretree one constructs a tree by `joining the dots' of the
pretree. In a subset of a tree there is a natural ternary
relation, if $a,b,c$ are 3 points at most one is between the 2
others. We use this betweeness relation to give a formal
definition of pretrees.

\begin{Def}
Let $\cP$ be a set and let $R\subset \cP\times \cP \times \cP$. We
say then that $R$ is a \emph{betweeness} relation. If $(x,y,z)\in
R$ then we write $xyz$ and we say that $y$ is between $x,z$. $\cP$
equipped with this betweeness relation is called a \emph{pretree}
if the following hold:

1. there is no $y$ such that $xyx$ for any $x\in \cP$.

2. $xzy \Leftrightarrow yzx$

3. For all $x,y,z$ if $y$ is between $x,z$ then $z$ is not between
$x,y$.

4. If $xzy$ and $z\ne w$ then either $xzw$ or $yzw$.

\end{Def}

\begin{Ex} The obvious example of a pretree is the vertex set of a tree. Note
also that any subset of a pretree is a pretree. Another example of
a pretree is the edge set of a tree. Not all pretrees are subsets
of trees. Indeed any linearly ordered set $(P,<)$ can be seen as a
pretree. 
\end{Ex}

\begin{Def} We say that a pretree $\cP$ is \textit{discrete} if for any
$x,y\in \cP$ there are finitely many $z\in \cP$ such that $xzy$.
\end{Def}

If there is no $z$ between $x,y\in \cP$ we say that $x,y$ are
adjacent.

Let $\cP$ be a countable discrete pretree. We recall briefly how can one pass
from $\cP$ to a tree (see \cite{B2} for a more general construction).

 We call a subset
$H\subset \cP$ a star if any $a,b\in H$ are adjacent. We define
now a tree $T$ as follows:

$$V(T)=\cP\cup \{\text {maximal stars of }\cP\}$$

$$E(T)=\{(v,H):v\in \cP, v\in H, H\text{ maximal star}\}$$

We show that $T$ is indeed a tree. Since $\cP$ is discrete $T$ is
connected. If $T$ contains a circuit then there are $x_1,...,x_n$
($n>2$) in $\cP$ such that $x_i$ is adjacent to $x_{i+1}$ and
$x_i$ is not adjacent to $x_{i+2}$ for all $i\in \Bbb Z _n$.

Since $x_i,x_{i+2}$ are not adjacent there is $y$ such that
$x_iyx_{i+2}$. If $y\neq x_{i+1}$ then either $x_iyx_{i+1}$ or
$x_{i+2}yx_{i+1}$ but both these are impossible. Hence
$x_ix_{i+1}x_{i+2}$ holds. We claim that $x_1x_{i-1}x_i$ holds for
all $i\leq n$. We argue by induction. Since $x_{i+1}\neq x_{i-1}$,
by 4, we have that $x_1x_{i-1}x_{i+1}$ holds. Since
$x_{i-1}x_ix_{i+1}$ holds for $n>i\geq 2$ by 4 either
$x_{1}x_ix_{i+1}$ or $x_{i-1}x_ix_{1}$ holds. However by induction
$x_1x_{i-1}x_i$ holds, so by 3, necessarily $x_{1}x_ix_{i+1}$. So
$x_{1}x_{n-1}x_{n}$ holds, contradicting our assumption that
$x_1,x_n$ are adjacent.

\begin{Ex} Note that `adding' the stars is necessary in order to get a tree from a pretree. Consider for example the case of a pretree $\cP$ consisting
of three mutually adjacent elements $x,y,z$. Then to get a tree one adds a new `star' vertex $w$ and joins it by edges to $x,y,z$.
\end{Ex}

\section{Cut sets}

As our objective in this section is to study the end structure of graphs we will restrict, without loss
of generality, to graphs that do not contain loops. Indeed if $\Gamma $ is a graph and $\Gamma '$ is the subgraph
of $\Gamma $ obtained from $\Gamma $ by erasing all loops of $\Gamma $ then there is an obvious 1-1 and onto map
from the ends of $\Gamma '$ to the ends of $\Gamma $. 

It will be convenient to replace edge cuts by cuts consisting of
midpoints of edges. We set up some notation: If $\Gamma =(V,E)$ is a graph then we have the incidence map
$\psi $ from the set $E$ of edges to the set of unordered pairs of vertices. So if $e$ is an edge $\psi (e)=\{v,u\}$ where
$v,u$ are the endpoints of $e$. In general $v=u$ is possible (when $e$ is a loop), however here since we assume that
$\Gamma $ has no loops, $v,u$ are distinct. If $K\subset V$ we say that $K$ is a vertex cut if the graph we obtain from $\Gamma $ by erasing all edges
incident to $K$ and with vertex set $V-K$ has more than one connected component. Abusing notation slightly we denote this new graph by
$\Gamma -K$. If $C$ is a component of $\Gamma - K$ we denote by $\partial C$ the set of vertices of $\Gamma $ which do not lie in $C$ and are incident to edges
that intersect $C$ (note that we may see $C$ as a subgraph of $\Gamma $). Finally we denote by $\bar C$ the graph obtained by adding to $C$ all edges that intersect $C$.
So the vertex set of $\bar C$ is $V(\bar C)=V(C)\cup \partial C$.

To replace edge cuts by vertex cuts we use the barycentric subdivision
of a graph:

\begin{Def} If $\Gamma $ is a graph the \textit{barycentric subdivision}
$\Gamma ^b$ of $\Gamma $ is the graph we obtain by subdividing
each edge of $\Gamma $ into two edges.
\end{Def}

More formally if $\Gamma =(V,E)$ then $\Gamma ^b=(V^b,E^b)$ where $V^b=V\cup E$
and $E^b=\{(e,v):e\in E, v\in \psi (e)\}$ where $\psi $ is the incidence function of the graph $\Gamma $.

If $K$ is a minimal cut of $\Gamma $ then $K'=K\cap V^b$ is a vertex cut of $\Gamma ^b$. We remark that if
$C$ is a component of $\Gamma ^b-K'$ then $K'=\partial C$.

If $K$ is a vertex cut we say that an end $e$ is \textit{contained} in a component $C$ of $\Gamma -K$
if for any ray $r$ representing $e$ almost all vertices of $r$ are contained in $C$.



\subsection{Equivalent Cuts}

\begin{Def} Let $\Gamma $ be a graph with more than one end. Let $K_1$, $K_2$ be
minimal cuts of $\Gamma $. We say that $K_1,K_2$ are
\textit{equivalent} if any two ends $e_1,e_2$ of $\Gamma $ are
separated by $K_1$ if and only if they are separated by $K_2$. We
write then $K_1\sim K_2$ and we denote the equivalence class of
$K_1$ by $[K_1]$.
\end{Def}


We would like to associate to a graph $\Gamma $, in a canonical
way, a cactus $\mathcal{C}$ which encodes the minimal cuts of the graph
$\Gamma $. To be more precise we will encode minimal cuts up to
equivalence. We will proceed by defining a pretree. There is a
natural way to define betweeness of equivalence classes of minimal
cuts, which we describe now. Let $\mathcal { E}$ be the set of ends of $\Gamma $.
If $K$ is a minimal cut then $K$ partitions $\mathcal { E}$, so
we may write $\mathcal {E}=K^{(1)}\cup K^{(2)}$ where 2 ends are separated by $K$ if and only if they lie
in different sets of this partition.
Clearly if $L$ is equivalent to $K$ then (after relabelling) $K^{(1)}=L^{(1)},\, K^{(2)}=L^{(2)}$.

\begin{Def} Let $\Gamma $ be a graph and let $K,L,M$ be inequivalent minimal cuts of
$\Gamma $. We say that $[L]$ is \textit{between} $[K],[M]$ if, possibly after relabelling,
$$K^{(1)}\subset L^{(1)}\subset M^{(1)}\, .$$
\end{Def}
Clearly if $[L]$ is between $[K],[M]$ then we have also $K^{(2)}\supset L^{(2)}\supset M^{(2)}$.

It is easy to see that the axioms 1,2,3 of the
pretree definition are satisfied. However axiom 4 is not satisfied
because of `crossing' cuts. We define formally crossing cuts in the next section
and we show that such cuts have some surprisingly simple
structure. This will allow us to remedy the problem of crossing
cuts and define a pretree.

We give in the next lemmas an equivalent way to define betweeness.

\begin{Lem}\label{subset} Let $\Gamma $ be a graph and let $K,L$ be inequivalent minimal cuts of
$\Gamma $ such that $K^{(1)}\subset L^{(1)}$. Then for any
$L_1\in [L]$ there is some $K_1\in [K]$ such that $K_1$
intersects a single connected component of $\Gamma -L_1$.
\end{Lem}

\begin{proof}

It will be convenient to replace $\Gamma $ by its
barycentric subdivision and view edge cuts as vertex cuts of the
barycentric subdivision. To keep notation simple we keep denoting
the barycentric subdivision by $\Gamma $.

Let $L_1\in [L]$. Let $C_1,C_2$ be the two connected components
of $\Gamma -L_1$, where the ends of $L^{(1)}$ are contained in $C_1$. We set:
$$ c=|K\cap L_1|,\,\, n=|K|=|L_1|$$

Consider $\Gamma - (L_1\cup K)$.
Clearly there are components $U_1,U_2$ of $\Gamma - (L_1\cup K)$ such that
$U_1$ contains all ends in $K^{(1)}$ and $U_2$ contains all ends in $L^{(2)}$.
We have then 
$$\bd U_1\subset K\cup L_1,\,\, \bd U_2\subset K\cup L_1,\,\, \bd U_1\cap \bd U_2\subset K\cap L_1 $$
So $$2n-c\geq|\bd U_1\cup \bd U_2|= |\bd U_1|+ |\bd U_2|-|\bd U_1\cap \bd U_2|\geq 2n- |\bd U_1\cap \bd U_2|$$
It follows that $|\bd U_1\cap \bd U_2|=c$ and $\bd U_1\cap \bd U_2= K\cap L_1$. Since all inequalities
are equalities $|\bd U_1|=|\bd U_2|=n$ and $\bd U_1\in [K]$. So $K_1=\bd U_1$ satisfies the requirements of the lemma.

\end{proof}

\begin{Lem} Let $\Gamma $ be a graph and let $K,L,M$ be inequivalent minimal cuts of
$\Gamma $. Then $[L]$ is between $[K],[M]$ if and only if for any
$L_1\in [L]$ there are $K_1\in [K], M_1\in [M]$ such that $K_1$
intersects a single connected component $C_1$ of $\Gamma -L_1$, $M_1$
intersects a single connected component $C_2$ of $\Gamma -L_1$ and $C_1\ne
C_2$.
\end{Lem}

\begin{proof}
As before we replace $\Gamma $ by its
barycentric subdivision and view edge cuts as vertex cuts of the
barycentric subdivision. 

Suppose that for some $L_1\in [L]$ there are $K_1\in [K], M_1\in [M]$ such that $K_1$
intersects a single connected component $C_1$ of $\Gamma -L_1$, $M_1$
intersects a single connected component $C_2$ of $\Gamma -L_1$ and $C_1\ne
C_2$. Set $L^{(1)}$ to be all ends of $\Gamma $ contained in $C_1$.
Let $C_1'$ be the connected component of $\Gamma -K_1$ contained in $C_1$ and let
$C_2'$ be the connected component of $\Gamma -M_1$ contained in $C_2$. Set 
$K^{(1)}$ to be all ends of $\Gamma $ contained in $C_1'$ and $M^{(2)}$ to be all ends of $\Gamma $ contained in $C_2'$.
Then clearly $K^{(1)}\subset L^{(1)}\subset M^{(1)}$, so $[L]$ is between $[K],[M]$.

Conversely now, assume that $[L]$ is between $[K],[M]$. We have then $K^{(1)}\subset L^{(1)}\subset M^{(1)}$. Let $L_1\in [L]$ and let $C_1,C_2$
be the connected components of $\Gamma -L_1$, where $C_1$ contains all ends in $K^{(1)}$. By lemma \ref{subset} there is $K_1\in [K]$ such that
$K_1\subset C_1$. Since $M^{(2)}\subset L^{(2)}$ by lemma \ref{subset} again there is $M_1\in [M]$ such that
$M_1\subset C_2$.

\end{proof}

\subsection{Crossing cuts}

\begin{Def} Let $\Gamma $ be a graph and let $K,L$ be minimal cuts of
$\Gamma $. We say that $[K]$ \textit{crosses} $[L]$ if $K^{(i)}\cap L^{(j)} \ne \emptyset $ for all $i,j=1,2$.

\end{Def}

\begin{Rem} Clearly if $K,L$ are inequivalent minimal cuts either $[K], [L]$ cross or, after relabelling, 
$K^{(1)}\subset L^{(1)}$. From lemma \ref{subset} we have that the following are equivalent:

a. $[K]$ crosses $[L]$.

b. For some
$L_1\in [L]$ any $K_1\in [K]$ intersects both connected components
of $\Gamma -L_1$.

c. For any
$L_1\in [L]$ any $K_1\in [K]$ intersects both connected components
of $\Gamma -L_1$.

\end{Rem}

\begin{Lem}\label{basic} Let $K,L$ be minimal cuts of a graph $\Gamma $.
If $[K]$ crosses $[L]$ then for any $L_1\in [L]$ and any $K_1\in
[K]$, the following hold:

a. $|K_1|=2k$ for some $k\in \Bbb N$.

b. $K_1\cap L_1=\emptyset $, 
and the intersections of $K_1$ with both components of $\Gamma
-L_1$ contain $k$ elements. 

c. $\Gamma -(K_1\cup L_1)$ has
exactly 4 connected components, and each of these components
contains at least one end of $\Gamma $.
\end{Lem}
\begin{proof} It will be convenient to replace $\Gamma $ by its
barycentric subdivision and view edge cuts as vertex cuts of the
barycentric subdivision. To keep notation simple we keep denoting
the barycentric subdivision by $\Gamma $.

We denote the connected components of
$\Gamma -L_1$ by $C_1,C_2$. Let's say that ends in $L^{(1)}$ are contained in $C_1$ and ends in
$L^{(2)}$ are contained in $C_2$. Since $[K],[L]$ cross there are ends
$e_{ij}\in K^{(i)}\cap L^{(j)} $ where $i,j=1,2$.

We denote the connected components of $\Gamma -K_1$ by $D_1,D_2$ where ends in $K^{(1)}$ are contained 
in $D_1$ and ends in $K^{(2)}$ are contained 
in $D_2$.

So $e_{ij}\in D_i\cap C_j$.

We set:
$$k_1=|K_1\cap C_1|,\, k_2=|K_1\cap C_2|,\,l_1=|L_1\cap D_1|,\,l_2=|L_1\cap
D_2|$$

We denote $n=|L_1|$ and $m=|L_1\cap K_1|$. Obviously
$$k_1+k_2+m=n,\,\, l_1+l_2+m=n$$

Let's pose further

$$m_1=|\partial (C_1\cap D_1)\cap K_1\cap L_1|$$
$$m_2=|\partial (C_1\cap D_2)\cap K_1\cap L_1|$$
We remark that
$$\partial (C_2\cap D_2)\cap K_1\cap L_1=\partial (C_1\cap D_1)\cap K_1\cap
L_1$$

$$\partial (C_1\cap D_2)\cap K_1\cap L_1=\partial (C_2\cap D_1)\cap K_1\cap
L_1$$

and $$m_1+m_2=m$$

We have also

$$|K_1|=k_1+k_2+m_1+m_2=n=|L_1|=l_1+l_2+m_1+m_2\,\,\, (*)$$

We remark that

$$|\partial (C_1\cap D_1)|=k_1+l_1+m_1$$
Since $e_{11}\in C_1\cap D_1$ and $\partial (C_1\cap D_1)$ separates
$e_{11}$ from $e_{12}$ we have that
$$k_1+l_1+m_1\geq n \,\,\,\,\, (1)$$
Considering similarly $C_2\cap D_1$, $C_1\cap D_2$ and $C_1\cap D_2$ we obtain
the inequalities:
$$k_2+l_1+m_2\geq n\,\,\,\,\, (2)$$
$$k_1+l_2+m_2\geq n \,\,\,\,\, (3)$$
$$k_2+l_2+m_1\geq n \,\,\,\,\, (4)$$

Adding up (1),(2),(3),(4) and using (*) we obtain
$$4n-2(m_1+m_2)\geq 4n$$

It follows that $m_1=m_2=0$. 

Further we have that necessarily (1), (2), (3), (4) are equalities.
From (1), (2) it follows that $k_1=k_2$. Similarly from (1), (3) we get that $l_1=l_2$.
This proves assertions a and b of the lemma.

Part c also follows since $|\bd (C_i\cap D_j)|=n$ for all $i,j=1,2$, $C_i\cap D_j$ is necessarily
connected, otherwise by considering its connected component containing $e_{ij}$ we would obtain an end
cut with less than $n$ edges, a contradiction.

%

%
%

\end{proof}

\begin{Cor} Let $\Gamma $ be a graph. If the cardinality of minimal cuts of $\Gamma $ is an odd number then
the set of minimal cuts with the betweeness relation defined
earlier is a discrete pretree. So the set of minimal cuts in this
case can be represented by a tree.
\end{Cor}


\subsection{Cyclic sets}


%
%
%
%

\begin{Def} Assume that the cardinality of a minimal cut of a
graph $\Gamma $ is $2k$. A subgraph $S$ of $\Gamma ^b$ is called
$k$-\textit{cyclic} if it is a union of $m\geq 4$ finite subgraphs, $S=S_{1}\cup S_{2}...\cup S_{m}$
 and there are connected subgraphs $M_{1}$,..., $M_{m}$ of $\Gamma ^b$
 so that for each $i\in \Bbb Z_m$:
\begin{enumerate}
   \item $S_i\cap M_i=\{s_{i1},...,s_{ik}\}, S_i\cap
   M_{i+1}=\{s_{i1}',...,s_{ik}'\}$ with $s_{ij},s_{ij}'$ in
   $V(\Gamma ^b)-V(\Gamma )$.

    \item $M_i\cup M_{i+1}$ separates $S_i$ from $\bigcup _{j\ne
    i}S_j$ and $S_i\cup S_{i-1}$ separates $M_i$ from $\bigcup _{j\ne
    i}M_j$

    \item $\bigcup (M_{i}\cup S_i)=\Gamma ^b$
    and for each $i$, $M_i$ contains at least one end of
    $\Gamma $.

\end{enumerate}

We will often simply say that $S$ is cyclic rather than $k$-cyclic.
We will say that the $M_i$'s  are the \textit{beads} and the $S_i$'s are the \textit{elements} of the cyclic
set $S$. 

\end{Def}

\begin{figure}[h]
\includegraphics[width=4.2in ]{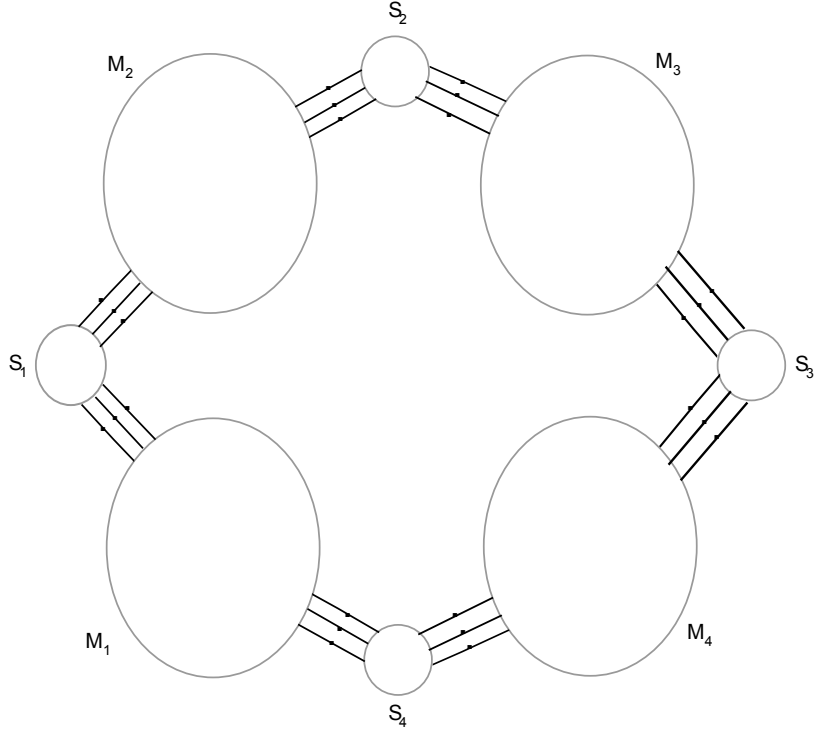}
\vspace{0.1in}
 \caption{A $3$-cyclic set with 4 beads}
\end{figure}

\begin{Lem}\label{union} Let $K_1,K_2$ be  minimal cuts such that $[K_1],[K_2]$ cross each other.
Then both $K_1,K_2$, are contained in a cyclic set $S$.
\end{Lem}
\begin{proof} As usual we see $K_1,K_2$ as vertex cuts of $\Gamma ^b$.
Let $C_1,C_2$ be the connected components of $\Gamma ^b
-K_1$ and let $D_1,D_2$ be the connected components of $\Gamma ^b
-K_2$. Recall that if $C$ is a connected component of $\Gamma ^b -K$ we denote by
$\bar C$ the graph which is the union of $C$ with all edges that intersect $C$. 

Then we can take $S=K_1\cup K_2$ and $$M_1= \bar C_1\cap \bar D_1$$
$$M_2= \bar C_1\cap
\bar D_2$$
$$M_3=\bar C_2\cap
\bar D_2$$
$$M_4=\bar C_1\cap
\bar D_1$$

By lemma \ref{basic} $S, M_1,M_2,M_3,M_4$ satisfy the definition
of a cyclic set.
\end{proof}

We recall the edge version of Menger's theorem (see eg. \cite{Bo}, thm. 7.17, p.170):
\begin{Menger}  Let $\Gamma $ be a graph and let $a,b$ be vertices of $\Gamma $. Then the 
maximum number of edge disjoint paths joining $a,b$ is equal to the minimum number of
edges in an edge cut separating $a,b$.
\end{Menger}

We say that a cyclic set $S=S_1\cup...\cup S_m$ is \textit{contained} in a cyclic set
$S'=S_1'\cup...\cup S_n'$ if for each $i=1,...,m$ there is some $j=1,...,n$ such that $S_i\subset S_j'$.
\begin{Lem}\label{max} Any $k$-cyclic set $S$ is contained in a
maximal $k$-cyclic set $\Sigma $.
\end{Lem}
\begin{proof} Let's say that $S=S_1\cup ...\cup S_m$ is a cyclic set of a graph $\Gamma $.
Let $M_2$ be a bead of $S$ and let $S_1\cap M_2=\{s_{1},...,s_{k}\}, S_2\cap
M_{2}=\{t_{1},...,t_{k}\}$ with $s_{i},t_{i}$ in
$V(\Gamma ^b)-V(\Gamma )$. We identify all vertices $s_{1},...,s_{k} $ to a single vertex $a$ and
all vertices $t_{1},...,t_{k} $ to a single vertex $b$ to obtain a new graph $\Gamma '$.
$\Gamma '-\{a,b\}$ has two connected components $C_1,C_2$ (where, say, $C_1$ is obtained from $M_2$ by the vertex
identifications indicated above).



Then no set of less than $k$ edges
separates $a,b$ in $C_1$ or $C_2$. Therefore, by the edge version
of Menger's theorem, there are edge disjoint simple paths
$p_1,...,p_k$ in $C_1$ and $q_1,...,q_k$ in $C_2$ such that for
each one of these paths one endpoint lies in $\{s_1,...,s_k\}$ and
the other lies in $\{t_1,...,t_k\}$. We lift the paths $p_1,...,p_k$ and 
$q_1,...,q_k$ to $\Gamma $ and we keep denoting them the same way.
We remark now that if $M_i$, $(i\neq 1, i\in \Bbb Z _k)$ is another bead of $S$ then at least one of
$(M_i\cap S_i)\cup \{s_1,...,s_k\}$, $(M_i\cap S_i)\cup \{t_1,...,t_k\}$
is a minimal cut. It follows that $M_i\cap S_i$ intersects each one of
$q_1,...,q_k$. Similarly $M_i\cap S_{i-1}$ intersects each one of $q_1,...,q_k$.

Let $S'=S_1'\cup...\cup S_n' $ be a cyclic set containing $S$. Clearly $n\geq m$ since
distinct elements of $S$ are contained in distinct elements of $S'$. This is because the elements
of $S'$ do not contain any minimal cuts.

Let $S_i'$ be an element of $S'$. For any $j\in \Bbb Z _n, j\neq i$, $$K_{ij}=(S_i'\cap M_i')\cup (S_{j}'\cap M_{j}')$$
is a minimal cut (where $M_i',M_j'$ are beads of $S'$). Choosing $j,p,r$ appropriately we may find a minimal cut:
$$L_{pr}=(S_p\cap M_p)\cup (S_r\cap M_r)$$ so that $K_{ij}, L_{pr}$ cross each other. Specifically if $S_i'$ contains come
$S_t$ we may pick $S_j'$ so that it contains $S_{t+2}$ and take $p=t+1,r=t+3$. If $S_i'$ contains no $S_t$ then there are
$i_1,i_2$ so that $i$ is between $i_1,i_2$, some $S_p$ is contained in $S_{i_1}'$ and $S_{p+1}$ is contained in $S_{i_2}'$.
Take then $r=p+1$ and $j$ so that $S_j'$ contains $S_{p+2}$.

 It follows that
$S_i'\cap M_i'$ intersects the union of the arcs $p_1\cup q_1\cup...\cup p_k\cup q_k$ ($i=1,...,k$) at $k$ points,
so it is contained in this union. By the same reasoning, the same holds for $S_i'\cap M_{i+1}'$. Since $S'$ is determined
by the sets $S_i'\cap M_{i+1}', S_i'\cap M_{i}'$ we have that there are finitely many cyclic sets containing $S$. Therefore
there is a maximal such set $\Sigma $ containing $S$. 
\end{proof}

\begin{Lem}\label{hyp-in} Let $\Sigma $ be a maximal cyclic set of $\Gamma $ 
and let $K$ be a minimal cut crossing some minimal cut contained in $\Sigma $. Then $K$ is equivalent to a minimal cut contained in $\Sigma $.
\end{Lem}
\begin{proof}
Let's say that $$\Sigma =S_1\cup...\cup S_m$$
Assume that $K$ is not equivalent to any minimal cut contained in $\Sigma $. Then there is some bead of $\Sigma $, say $M_i$ such that
$K$ separates some ends of $M_i$. If $$K'=(M_i\cap S_i)\cup (M_i\cap S_{i-1})$$
then $K$ crosses $K'$ so by lemma \ref{basic} $K=K_1\cup K_2$, $K_1$ is contained in $M_i$, $K_2$ is contained in $$\bigcup _{j\neq i}M_j$$
and $|K_1|=|K_2|$. Let $$A=M_i\cap S_i,\,B=M_i\cap S_{i-1} $$
We claim that $K$ separates each vertex of $A$ from each vertex of $B$. 
We distinguish two cases. If $K$ separates some ends of some bead $M_j$, $j\neq i$ then $K_2\subset M_j$.
It follows that $K$ does not intersect at least one of the beads $M_{i-1}, M_{i+1}$. Let's say it does not intersect $M_{i-1}$.
If $M_{i-1}\cap S_{i-1}=C$ then by Menger's lemma each vertex of $B$ is connected to some vertex of $C$ by a path that does not intersect $K$. It follows that
all vertices of $B$ are connected in the same component of $\Gamma - K$. By lemma \ref{basic} all vertices of $A$ are contained in the
same component of $\Gamma - K$ too. If $K$ does not intersect $M_{i+1}$ we argue similarly using $A$ rather than $B$.

Assume now that $K$ does not separate ends of any $M_j$ with $j\neq i$. Then for some $j\neq i,i-1$ $K$ separates all ends in $M_j$ from
all ends in $M_{j+1}$. It follows that $K$ crosses the cut

$$L=(M_{j+1}\cap S_{j+1})\cup (M_j\cap S_{j-1}) $$
If $L_1=M_{j+1}\cap S_{j+1}$ then, by Menger's lemma, each vertex in $L_1$ is connected by a path that does not
intersect $K$ to some vertex of $B$. It follows that
all vertices of $B$ are contained in the same component of $\Gamma - K$. By lemma \ref{basic} all vertices of $A$ are contained in the
same component of $\Gamma - K$ too.

We see then that in both cases $K_1$ separates $A$ from $B$ in $M_i$.

It follows that we can enlarge $\Sigma $ by adding $K_1$. We obtain a cyclic set
$$\Sigma '=S_1\cup...\cup S_{i-1}\cup K_1\cup S_i\cup ...\cup S_m $$
This contradicts the maximality of $\Sigma $.
\end{proof}

\begin{Lem}\label{unique} Let $\Sigma $ be a maximal cyclic set of $\Gamma $ containing a given $k$-cyclic set $S$ 
and let $K$ be a minimal cut which crosses a minimal cut contained in $\Sigma $. Then $K$ is contained in $\Sigma $.
In particular $S$ is contained in  a unique maximal cyclic set.
\end{Lem}
\begin{proof}

Let's say that $$\Sigma =S_1\cup...\cup S_m$$
By lemma \ref{hyp-in} $K$ is equivalent to $K'$ where 
$K'$ is contained in a union $S_i\cup S_j$ for some $i,j$, $j\neq i, i+1$. 
We set $$S_{i-1}\cap M_i=\{a_1,...,a_k\},\,S_{i+1}\cap M_{i+1}=\{b_1,...,b_k\}$$
$$S_{j-1}\cap M_j=\{a_1',...,a_k'\},\,S_{j+1}\cap M_{j+1}=\{b_1',...,b_k'\}$$

We remark that both minimal cuts $$L_1=\{a_1,...,a_k,b_1,...,b_k\},\, L_2=\{a_1',...,a_k',b_1',...,b_k'\}$$
cross $K$. So by lemma \ref{basic} $K$ can be written as disjoint union of two sets, $K_1,K_2$ where
$K_1$ is contained in the connected component $C$ of $\Gamma -L_1$ containing $M_i$ and 
$K_2$ is contained in the connected component $D$ of $\Gamma -L_2$ containing $M_j$. 

Applying Menger's theorem as earlier we see that there are edge disjoint paths $p_1,...,p_k$ in $C$
and $q_1,...,q_k$ in $D$
such that $p_i$ joins $a_i$ to $b_i$, and $q_i$ joins $a_i'$ to $b_i'$ for all $i=1,...,k$ (this of course up to
relabeling of the $a_i$'s, $b_i$'s). Since $K$ separates $M_i$ from $M_{i+1}$ and $M_j$ from $M_{j+1}$ $K_1$ intersects
each $p_i$ at one point and $K_2$ intersects each $q_i$ at one point. 

Let $M_i', M_{i+1}'$ be the infinite components of $\Gamma -(S_{i-1}\cup S_i\cup K)$ contained in $M_i$, $M_{i+1}$ respectively.
Then $$\partial M_i'\subset K_1\cup S_i\cup S_{i-1}$$
Clearly $$S_{i-1}\cap M_i\subset \partial M_i'$$
We remark now that if $\partial M_i'$ intersects a path $p_t$ in 3 points then the first point is $a_t$ the second lies on $K_1$ and the third point say $c$ lies on $S_1$.
However in this case there is a path joining $b_t$ to $a_t$ which does not intersect $K$. Indeed take $p_t$ from $b_t$ to $c$ and then continue
with a path in $M_i'$ joining $c$ to $a_i$. This is clearly a contradiction since $K$ separates $a_t,b_t$.
It follows that $\partial M_i'$ intersects each $p_t$ at at most 2 points, and since $\partial M_i'$ has at least $2k$ points we conclude
that $\partial M_i '$ intersects each $p_t$ at exactly 2 points. We argue similarly for $M_{i+1}'$.
We conclude that if $S_i'$ is the union of the finite connected components of $$\Gamma -(M_i'\cup M_{i+1}')$$
then $S_i\subset S_i'$ and $|S_i'\cap M_i'|=k$ and $|S_i'\cap M_{i+1}'|=k$. So we may replace $S_i$ by $S_i'$ in the cyclic set $\Sigma $. This contradicts the
maximality of $\Sigma $ unless $S_i'=S_i$ and $K_1\subset S_i$. Arguing similarly for $K_2$ we have that
$K_2\subset S_j$, so $K$ is contained in $\Sigma $.

We show now that $\Sigma $ is unique. Let $\Sigma _1, \Sigma _2$ be two maximal cyclic sets containing $S$.
Then for any element $S_i$ in $\Sigma _1$ there is some $S_j\in \Sigma _1$ such that the cuts
$$K_1=(S_i\cap M_i)\cup (S_i\cap M_i)$$

$$K_1=(S_i\cap M_{i+1})\cup (S_i\cap M_i)$$
cross some minimal cut contained in $\Sigma _2$. By lemma \ref{hyp-in} and the proof above it follows that $K_1,K_2$ are contained in $\Sigma _2$. This implies
that for all $i$, $S_i$ is contained in $\Sigma _2$ so $\Sigma _1\subset \Sigma _2$. By symmetry $\Sigma _2\subset \Sigma _1$ so $\Sigma _1=\Sigma _2$.

\end{proof}


\begin{Cor} If a minimal cut $K$ crosses some other minimal cut $L$ then
every minimal cut $K'\in [K]$ is contained in a cyclic set.
\end{Cor}

\begin{proof} By lemmas \ref{union} and \ref{max} $K,L$ are contained in a maximal cyclic set $S$.
 $K'$ crosses $L$, so by lemmas \ref{hyp-in},\ref{unique} $K'$ is contained in $S$ too.
\end{proof}

\begin{Def} We say that a minimal cut $K$ is \textit{isolated} if it does not cross any other minimal cut.
\end{Def}

We can now define a pretree $\cP$ `encoding' all minimal cuts of $\Gamma $. The elements
of $\cP$ are the maximal cyclic sets of $\Gamma $ and the equivalence classes of the isolated minimal
cuts of $\Gamma $. We make now some observations that will allow us to define betweeness in $\cP$.

Note that if $S=S_1\cup...\cup S_n$ is a maximal cyclic set with beads $M_1,...,M_n$ then each end of $\Gamma $ is contained in
exactly one of the $M_i$'s. So $S$ partitions the set of ends $$\mathcal {E}=M^{(1)}\cup...\cup M^{(n)}$$ where we denote
by $M^{(i)}$ the set of ends contained in $M_i$. If $K$ is a minimal cut not contained in $S$ such that $K$ separates
some ends $e_1,e_2$ that lie in some $M^{(i)}$ then $K$ many not separate any two ends $e_1',e_2'$ that lie in some
$M^{(j)}$ with $i\neq j$. Indeed in that case $K$ would cross the cut $(M_i\cup S_i)\cup (M_i\cup S_{i-1})$ so by lemma \ref{hyp-in} $K$ would
be contained in $S$. We remark further that if $K$ does not separate ends that lie in any $M^{(i)}$ then there is some $i$ such that
$K$ separates all ends in $M^{(i)}$ from all ends in $M{(j)}$ for all $j\neq i$. Indeed if not, as before, $K$ crosses some minimal cut contained in $S$, so
$K$ lies in $S$. We conclude that in all cases the following holds: if $K$ is a minimal cut not contained in $S$ and if we denote by
$K^{(1)}\cup K^{(2)}$ the partition of ends of $\Gamma $ induced by $K$ then $K^{(1)}$ or $K^{(2)}$ is contained in some $M^{(i)}$.
By lemma \ref{subset} it follows further that if $K$ is a minimal cut that is not contained in $S$ then it is equivalent to a cut
$K'$ such that $K'\subset M_i$ for some $i$.

Let $S'=S_1'\cup...\cup S_n'$ be another maximal cyclic set. Then for each minimal cut $K$ contained in $S'$ there is an $i$ such that
$K$ is equivalent to a minimal cut contained in $M_i$. However if there are minimal cuts $K,L$ in $S'$ such that, say $K$ is equivalent to
a minimal cut in $M_i$ and $L$ is equivalent to a minimal cut in $M_j$ with $j\neq i$ then there is a minimal cut in $S$ that crosses
a minimal cut in $S'$. But this implies that $S=S'$. It follows that all minimal cuts in $S'$ are equivalent to cuts that are contained in a single bead
$M_i$ of $S$.
%

If $S_1,S_2,S_3$ are distinct elements of $ \cP$ we define betweeness as follows: If $S_1$ is cyclic, then it is between $S_2,S_3$ if the minimal cuts in $S_2,S_3$
are equivalent to cuts which are contained in distinct beads of $S_1$.
If $S_1=[K_1]$ is not cyclic then $S_1$ is between $S_2,S_3$ if all minimal cuts $ K_2$ in $S_2$, $K_3$ in $S_3$
are equivalent to cuts that lie in distinct components of $\Gamma -K_1$.

%
\begin{Thm} \label{final} $\cP$ with the betweeness relation defined above is a pretree.
\end{Thm}
\begin{proof} Axioms 1 and 2 of the pretree definition obviously hold.
We show that axiom 3 holds. Assume that $S_1$ is between $S_2,S_3$. We will show that $S_3$ is not between $S_1,S_2$.
Assume first that $S_1$ is cyclic.
Since $S_1$ is between $S_2,S_3$ then the minimal cuts in $S_2,S_3$ are equivalent to minimal cuts which are contained in distinct beads, say $M_i,M_j$ of $S_1$.
This implies that for any minimal cut in $S_3$ the minimal cuts in $S_1,S_2$ are equivalent to minimal cuts that are contained in the same bead of $S_3$ if
$S_3$ is cyclic or in the same component of $\Gamma -K$ for $K\in S_3$ if $S_3$ is not cyclic. In both cases $S_1S_3S_2$ does not hold.
Assume now that $S_1=[K]$ where $K$ is an isolated cut. Then the minimal cuts in $S_2, S_3$ are equivalent to cuts that lie in distinct components
of $\Gamma - K$. Then, if $S_3$ is cyclic, the minimial cuts in $S_1, S_2$ are equivalent to minimal cuts that lie in the same bead of $S_3$ so $S_1S_3S_2$ does not hold.
Similarly in $S_3=[K]$ all minimal cuts in $S_1,S_2$ are equivalent to minimal cuts that lie in the same component of $\Gamma -K$. So $S_1S_3S_2$ does not hold
in this case either.

%

We show finally axiom 4. Assume that $S_2S_1S_3$ holds and that $S_4\neq S_1$. If $S_1$ is cyclic then then the minimal cuts in $S_2,S_3$ are equivalent to cuts contained in distinct
beads of $S_1$. If all minimal cuts in $S_4$ are equivalent to cuts contained in the same bead as $S_2$ then $S_3S_1S_4$ holds. Otherwise $S_2S_1S_4$ holds.
If $S_1$ is an equivalence class of an isolated cut $K$ then the minimal cuts in $S_2,S_3$ are equivalent to cuts contained in distinct components
of $\Gamma -K$. If the minimal cuts in $S_4$ are equivalent to minimal cuts contained in the same component of $\Gamma -K$ as the minimal cuts of $S_2$ then $S_3S_1S_4$ holds. Otherwise $S_2S_1S_4$ holds. This shows that axiom 4 is satisfied.

\end{proof}

Clearly the pretree $\cP$ is discrete, so it can be completed to a tree $T$ encoding all minimal cuts of $\Gamma $.
We give now a detailed description of $T$. The vertices of $T$ are of three types: isolated cuts, cyclic sets and `star' vertices (see section 2 for the
`star' vertices). We remark that if $S$ is a cyclic set in $\cP$, $S$ is adjacent to the isolated cuts that correspond to its beads.
So if $S=S_1\cup ...\cup S_n$ and $M_i$ is a bead of $S$ then $S$ is adjacent to the isolated cut $(M_i\cap S_{i-1})\cup (M_i\cap S_{i})$.
It follows that all star vertices adjacent to $S$ have degree 2. If $K$ is an isolated cut of $\cP$ and $H$ is a star of $\cP$ containing $K$ then
either $H$ consists of a cyclic set $S$ and $K$ or $H$ consists of at least 3 isolated cuts.

To retain the cyclic structure of the crossing cuts we replace $T$ by a cactus $\mathcal{C}$ as follows:
A cyclic set $S$ of $\cP$ gives a cycle $C$ of $\mathcal{C}$ with vertices corresponding to the beads of the cyclic set $S$.
Each vertex of $C$ is joined to the corresponding star vertex. In this way we obtain a cactus. We further simplify this cactus as follows:
If $K$ is an isolated cut adjacent to a cyclic set $S$ in $\cP$ then $K$ is joined to $S$ by a path of two edges (from $K$ to the star vertex and then
from the star vertex to $S$). We collapse all these 2-edge paths joining isolated cuts to cyclic sets.


This is because such isolated cuts are already represented in the cycle (by the two edges adjacent
to the bead).  For symmetry's sake finally we `double' all separating edges of the cactus. Clearly now we
have a 1-1 correspondence between the minimal cuts of $\Gamma $ and the minimal edge cuts of $\mathcal{C}$. We state this formally:

\begin{Thm}\label{cactus} Let $\Gamma $ be a graph. Then there is a cactus $\mathcal{C}$, an onto map $f$ from the ends
of $\Gamma $ to the union of the ends of $\mathcal{C}$ with the end vertices of $\mathcal{C}$ and a 1-1 and onto map $g$ from equivalence classes
of minimal cuts of $\Gamma $ to the minimal edge cuts of $\mathcal{C}$ so that two ends $e_1,e_2$ of $\Gamma $ are separated by a minimal
cut $K$ if and only if $f(e_1),f(e_2)$ are separated by $g([K])$. Moreover any automorphism of $\Gamma $ induces an automorphism of $\mathcal{C}$.
\end{Thm}

\begin{Cor} (Stallings end theorem) Let $G$ be a group acting transitively on a graph $\Gamma $ with more than 2 ends. Then
$G$ splits as $G=A*_FB$ or $G=A*_F$ where $F$ has a finite index subgroup which is a stabilizer of an edge of $\Gamma $.
\end{Cor}
\begin{proof} We associate a tree $T$ to $\cP$. The action is non trivial since the action of $G$ on $\Gamma $ is transitive. It follows that $G$ splits over a stabilizer of an edge.
Edges correspond to equivalence classes of minimal cuts. So edge stabilizers stabilize equivalence classes of minimal cuts. Since such
equivalence classes contain finitely many edges the result follows.
\end{proof}
Stallings' theorem covers the 2-ended case too. However our cactus is this case reduces to a single point. 
We remark that if $G$ is a finitely generated group and $\Gamma $ its Cayley graph
the 2-ended case is simpler and it is easy to show that in this case $G$ has a finite index subgroup isomorphic to $\Bbb Z$.
We note finally that Kr\"{o}n \cite{Kr} has given recently a very elegant proof of Stallings theorem using the methods of \cite{DK}.

\section{Generalizations}\label{gen}

One can show that the `cactus structure' of cuts exists in other settings as well.
We discuss here some such generalizations, which are interesting for finite graphs.

\begin{Def} 
Let $\Gamma $ be a graph and $K$ a set of edges of $\Gamma $.
We say that $K$ is an $n$-\textit{cut} if $\Gamma -K$ has more than one component and
$|K|=n$. 
\end{Def}
\begin{Def} Let $\Gamma $ be a graph and
let $S$ be a set of vertices of $\Gamma $. We call $S$
\textit{$n$-inseparable} if for any $r$-cut $K$, with $r\leq n$,
$S$ is contained in a single component of $\Gamma -K$.
\end{Def}

We remark that Dunwoody and Kr\"{o}n \cite{DK} consider a similar
notion of inseparable sets but their definition is slightly
stronger, they require further that $|S|>n$.

\begin{Def} Let $\Gamma $ be a graph and $k\in \Bbb N$. We define $N(k)$ to
be the smallest $n$ such that there are at least two distinct
maximal $n$-inseparable subsets of $\Gamma $ with at least $k$
vertices each. If there is no such $n$ we set $N(k)=\infty $. If for
some $k$, $N(k)<\infty $ we say that $\Gamma $ is a
\textit{$k$-thin} graph. We call an $N(k)$-cut $K$
\textit{essential} if both components of $\Gamma -K$ contain some
$N(k)$-inseparable set.
\end{Def}

We remark that if $k_1>k_2$ then $N(k_1)\geq N(k_2)$. In
particular if a graph is $k$-thin for some $k>2$ then it is also
2-thin. Clearly every graph with at least two vertices is 1-thin.
Assume that $\Gamma $ is $k$-thin, and set $n=N(k)$. If $S_1,S_2$ are $n$-inseparable subsets
and $K$ is an $n$-cut which separates $S_1,S_2$ then $\Gamma -K$ has exactly 2 components.

If $K_1,K_2$ are $n$-cuts of $\Gamma $ we say that that $K_1,K_2$ are equivalent if for
any two $n$-inseparable subsets of $\Gamma $, $S_1,S_2$ are separated by $K_1$ if and only if they
are separated by $K_2$. We denote the equivalence class of $K_1$ by $[K_1]$.

Lemma \ref{basic} applies in this context as well and one can show exactly as in the case
of minimal end cuts that all equivalence classes of $N(k)$-cuts of a $k$-thin graph are encoded by a cactus.

Clearly every graph with at least 2 vertices is $1$-thin. In this case $N(1)$ is the cardinality of a minimal edge cut
and every equivalence class has a single element. So this case amounts to the classical cactus theorem of
Dinits-Karzanov-Lomonosov (\cite{DKL}).

%

\begin{Def}

Let $\Gamma $ be a graph and $K$ a set of edges of $\Gamma $. 
We say that $K$ is an $(n,k)$-\textit{cut} if $|K|=n$ and 
$\Gamma -K$ has at least 2 components which have each at least $k$ vertices.
Let $S$ be a set of vertices of $\Gamma $ containing at least $k$
elements. We call $S$ \textit{$(n,k)$-inseparable} if for any
$(r,k)$-cut $K$, with $r\leq n$, $S$ is contained in a single
component of $\Gamma -K$.
\end{Def}

We note that $(n,k)$-cuts have been studied extensively in network theory
(see e.g. \cite{FaF}, \cite{ZY})

\begin{Def} Let $\Gamma $ be a graph. We say that $\Gamma $
is \textit{$(n,k)$-large} if $\Gamma $ has at least 2 distinct
maximal $(n,k)$-inseparable subsets. We set $M(k)$ to be the least
$n$ such that $\Gamma $ has at least two distinct maximal
$(n,k)$-inseparable subsets. If $M(k)<\infty $ we say that $\Gamma
$ is $k$-\textit{slim}.
\end{Def}

\begin{Def} Let $\Gamma $ be a $k$-slim graph. Let $K_1$, $K_2$ be
$(M(k),k)$-cuts of $\Gamma $. We say that $K_1,K_2$ are equivalent if
for any two $(n,k)$-inseparable subsets, $S_1,S_2$ of $\Gamma $ we
have that $S_1,S_2$ are contained in distinct components of
$\Gamma -K_1$ if and only if they are contained in distinct
components of $\Gamma -K_2$. We write then $K_1\sim K_2$ and we
denote the equivalence class of $K_1$ by $[K_1]$.
\end{Def}

\begin{Def} We call an $(M(k),k)$ cut
\textit{essential} if both components of $\Gamma -K$ contain some
$(M(k),k)$-inseparable set. 
\end{Def}

Lemma \ref{basic} applies to equivalence classes of essential  $(M(k),k)$-cuts,
so such cuts are also encoded by a cactus.

\end{document}
\bye